\documentclass[reqno, 11pt]{amsart}

\usepackage[top=3cm, bottom=3cm, left=3.2cm, right=3.2cm]{geometry}
\usepackage{amsmath}
\usepackage{amssymb,amsmath,amsthm,graphicx} 
\catcode`!=11
\let\!int\int \def\int{\displaystyle\!int}
\let\!lim\lim \def\lim{\displaystyle\!lim}
\let\!sum\sum \def\sum{\displaystyle\!sum}
\let\!sup\sup \def\sup{\displaystyle\!sup}
\let\!inf\inf \def\inf{\displaystyle\!inf}
\let\!cap\cap \def\cap{\displaystyle\!cap}
\let\!max\max \def\max{\displaystyle\!max}
\let\!min\min \def\min{\displaystyle\!min}
\let\!frac\frac \def\frac{\displaystyle\!frac}
\catcode`!=12

\newtheorem{theorem}{\bf Theorem}[section]
\newtheorem{lemma}{\bf Lemma}[section]

\newtheorem{proposition}{\bf Proposition}[section]

\def\pd#1#2{\frac{\partial#1}{\partial#2}}

\catcode`@=11 \@addtoreset{equation}{section} 
\catcode`@=12

\begin{document}

\title[Wavefronts for a nonlinear nonlocal bistable equation]{
Wavefronts for a nonlinear nonlocal bistable reaction-diffusion equation in population dynamics}
\thanks{Jing Li is supported by the Beijing Natural Science Foundation(Grant no. 9152013) and the
Research Grant Funds of Minzu University of China. This work is partially supported by DFG Project CH 955/3-1 and the DAAD project ``DAAD-PPP VR China''(Project-ID: 57215936).}

\author{Jing Li}
\address{College of Science, Minzu University of China, Beijing, 100081, P.R. China}
\email{{\tt matlj@163.com}}

\author{Evangelos Latos}\thanks{Corresponding author: Evangelos Latos}
\address{Lehrstuhl f\"ur Mathematik IV, Universit\"at Mannheim, 68131, Germany}
\email{evangelos.latos@math.uni-mannheim.de}

\author{Li Chen}
\address{Lehrstuhl f\"ur Mathematik IV, Universit\"at Mannheim, 68131, Germany}
\email{chen@math.uni-mannheim.de}

\date{}

\maketitle

\begin{abstract}
The wavefronts of a nonlinear nonlocal bistable reaction-diffusion equation,
\begin{align*}
\pd ut=\frac{\partial^2u}{\partial x^2}+u^2(1-J_\sigma*u)-du,\quad(t,x)\in(0,\infty)\times\mathbb R,
\end{align*}
with 
$
J_\sigma(x)=(1/\sigma)= J(x/\sigma)
$
and
$
\int_{\mathbb R} J(x)dx=1
$
are investigated in this article.
It is proven that there exists a $c_*(\sigma)$ such that for all $c\geq c_*(\sigma)$, a monotone
wavefront $(c,\omega)$ can be connected by the two positive equilibrium points.
On the other hand, there exists a $c^*(\sigma)$ such that the model
admits a semi-wavefront $(c^*(\sigma),\omega)$ with $\omega(-\infty)=0$. Furthermore, it is shown that
for sufficiently small $\sigma$,
 the semi-wavefronts are in fact wavefronts connecting $0$ to the largest equilibrium. In addition,
the wavefronts converge to those of the local problem as $\sigma\to0$.
\end{abstract}

\vskip5mm

{\small

}
\vskip5mm

\section{Introduction}
In this work we study the nonlinear nonlocal reaction-diffusion equation
\begin{align}
\pd ut=\frac{\partial^2u}{\partial x^2}+u^2(1-J_\sigma*u)-du\quad\hbox{in}\;(0,\infty)\times\mathbb R,
\label{1.1}
\end{align}
where $0\leq d<\frac29$, $J_\sigma(x)=\frac1\sigma J(\frac{x}{\sigma})$ is a $\sigma$-parameterized nonnegative
kernel with
$$
J\in L^1(\mathbb R),\quad\int_{\mathbb R} J(x)dx=1
$$
and
$$J_\sigma*u(x)=\int_{\mathbb R}J_\sigma(x-y)u(y)dy. $$
This equation has three constant solutions,
$$0, \quad a=\frac12(1-\sqrt{1-4d}),\quad A=\frac12(1+\sqrt{1-4d}).$$

The problem arises in population dynamics with nonlocal consumption of resources, for example in \cite{4,1}.
It is used to model the behavior of  various biological phenomena such as emergence and evolution of biological
species and the process of speciation. Actually, similar nonlocal structure in the reaction term
appears also in describing the behavior of cancer cells with therapy as well as polychemotherapy and
chemotherapy \cite{20,21}.

The reaction term $u^2(1-J_\sigma*u)-du$ consists of the reproduction which
is proportional to the square of the density, the available resources and the mortality.
The nonlocal consumption of the resources $J_\sigma*u(x)$ describes that the consumption at the space point $x$
is determined by the individuals located in some area around this point, where $J_\sigma$ represents the
probability density function that describes the distribution of individuals.

For $J(x)=1$, with a general nonlinearity, $u^\alpha(1-\int u(x,t)dx)$ in the multi-dimensional case,
the problem has been studied \cite{BCL,BC} in terms of the existence of the classical solutions both in bounded
and unbounded domains correspondingly. In \cite{CLP}, it is shown that the blow-up of the solution could happen
for some $\alpha>1$. However, whether the solution exists is still not known in one dimension when $\alpha=2$ .

In the case of $J(x)=\delta(x)$, where $\delta(x)$ is the Dirac function, equation \eqref{1.1} becomes the
so called Huxley equation, which is a classical reaction-diffusion equation.
It has the same constant solutions, $0$, $a$ and $A$ to the nonlocal problem.
The existence of traveling waves has been studied extensively in the literature
(see \cite{9,10,11,12,13,14} among others).  It's proved that there exists a minimum speed such that
the traveling waves
connecting $a$ and $A$ exist for all values of the speed greater than or equal to this minimum speed. While
the traveling waves connecting $0$ and $A$ exist only for a single value of the speed.

Compared to the rich results for the local version of the Fisher-KPP reaction diffusion equation,
very limited theoretical results exist for its nonlocal version. In the last few years,
there has been several works on wavefronts for some typical nonlocal reaction
diffusion equations. In the research of wavefronts,
in order to get a priori bounds for the existence and monotonicity properties of the fronts,
the classical methods  substantially depend on the application of comparison
principle. However, for the equation with nonlocal competition term, the most challenging point
arises from the lack of the comparison principle.
One first example is the following nonlocal
Fisher-KPP equation
\begin{align}
\pd ut=\frac{\partial^2u}{\partial x^2}+u(1-J_\sigma*u)\quad\hbox{in}\;(0,\infty)\times\mathbb R.
\label{1.7}
\end{align}
Berestycki {\it et al} \cite{4} proved that \eqref{1.7} admits a semi-wavefront connecting $0$ to an
unknown positive state for all $c\geq c^*=2$ and there is no such kind of wavefront with wave speed $c<2$.
In \cite{5}, Nadin {\it et al} numerically verified the existence of monotone wavefronts.
After that, Alfaro {\it et al} \cite{7} rigorously proved that \eqref{1.7} admits the rapid wavefront
connecting $0$ and $1$. Furthermore, Fang {\it et al} \cite{6} gave a sufficient and necessary condition for the existence of
monotone wavefronts of \eqref{1.7} that connect the two equilibrium points $0$ and $1$. In a recent paper by Hasik {\it et al}
\cite{15}, for nonsymmetric interaction kernel $J_\sigma$, the different roles of the right and the left
interactions are investigated.
Nonlocal equations with bistable reactions have been investigated in \cite{17,8,3}. In \cite{17},
 Wang {\it et al}
studied
\begin{align}
\pd ut=\frac{\partial^2u}{\partial x^2}+g(u,J*S(u)),
\end{align}
where $g(u,J*S(u))$ satisfies some bistable assumptions.
Although it is a nonlocal problem, due to their special assumptions, the comparison principle still holds.
Therefore,
by constructing various pairs of super- and sub-solutions, employing the comparison
principle and the squeezing technique, the authors proved the existence of monotone traveling wavefronts.

There are further results on equations with other bistable reactions, where comparison principle can not be applied.
In \cite{8}, Alfaro {\it et al.} considered the following equation
\begin{align}
\pd ut=\frac{\partial^2u}{\partial x^2}+u(u-\theta)(1-J_\sigma*u)\quad\hbox{in}\;(0,\infty)\times\mathbb R
\label{14}
\end{align}
with $0<\theta<1$. The Leray-Schauder degree method is used to indicate that
\eqref{14} admits semi-wavefronts connecting $0$ to an unknown positive steady state, which is
above and away from the intermediate equilibrium. For focusing kernel, it is proved that the wave connects
$0$ and $1$.

The wavefront solution $\omega(x-ct)$ for Equation \eqref{1.1} has been investigated, for small $\sigma$,
in \cite{3} by Apreutesei {\it et al.}.  It satisfies
\begin  {align}
\omega''(\xi)-c\omega'(\xi)+\omega^2(\xi)(1-J_\sigma*\omega(\xi))-d\omega(\xi)=0.\label{15}
\end{align}
They proved the existence of wavefronts of \eqref{1.1} that connect $0$ and $A$.
In fact, for small $\sigma$, the nonlocal
operator is a perturbation of the corresponding local operator, thus the implicit function theorem
can be applied.
More precisely, under the assumptions
$$\int_\mathbb R|z|J(z)dz<\infty,\quad\int_\mathbb R|z|^2J(z)dz<\infty,$$
they obtained that there exists $\sigma_0>0$ such that, for any $|\sigma|<\sigma_0$, equation \eqref{1.2} has a solution
$(c, \omega)\in C^{2+\alpha}(\mathbb R)\times\mathbb R$ with $\omega(-\infty)=0$ and
$\omega(+\infty)=A$. Furthermore, the solution is of the class $C^1$
with respect to $\sigma$.

In this paper, we study the existence of wavefronts of \eqref{1.1} which connect $a$ to $A$ and $0$ to $A$
respectively by using a totally different method from \cite{3}. The main results we obtained in this paper are as follows.

The first result shows the existence of wavefronts connecting $a$ to $A$ for any $\sigma$ with big enough wave
speed $c$.
\begin{theorem} Suppose $0\leq d <\frac{2}{9}$, then it holds that
\begin{enumerate}
\item[(i)] for any $\sigma>0$, there exists a $c_*(\sigma)>0$ such that when
$c\geq $ $\max$ $\{2\sqrt{2A-d}$, $c_*(\sigma)\}$, \eqref{1.1} admits a monotone
wavefront $\omega \in C^2(\mathbb R)$, i.e., $(c, \omega)$ is the solution of the following problem
\begin{align}
& \omega''-c\omega'+\omega^2(1-J_\sigma*\omega)-d\omega=0 \qquad
	 \hbox{in}\; \mathbb R,\nonumber \\
	 & \omega(-\infty)=a,\quad\omega(+\infty)=A.\label{1.2}
\end{align}
\item[(ii)] as $\sigma\to0$, $c_*(\sigma)$ converges to
$c_*$. Moreover, for any $c\geq \max\{2\sqrt{2A-d},c_*\}$, by fixing $\omega(0)=\frac12$,
$\omega$ has a subsequence converging to $\omega_0$ in $C^{1,\alpha}_{loc}(\mathbb R)$,
 where $(\omega_0,c)$ is the solution of the following problem
\begin{align}
	 & \omega_0''-c\omega_0'+\omega_0^2(1-\omega_0)-d\omega_0=0 \qquad
	 \hbox{in}\; \mathbb R, \nonumber\\
	 & \omega_0(-\infty)=a,\quad \omega_0(0)=\frac12,\quad  \omega_0(+\infty)=A.
	 \label{21}
\end{align}
\end{enumerate}
\end{theorem}

The second result demonstrates the existence of a semi-wavefronts connecting $0$ to an intermediate state $d_0$
for any $\sigma$; and furthermore this semi-wavefront can be extended to $A$ as $x$
goes to $+\infty$ in the case of small $\sigma$.
\begin{theorem}Suppose $0< d <\frac{2}{9}$, then it holds that
\begin{enumerate}
\item [(i)] there exists an $M>0$ such that for any $\sigma>0$ and $0<d_0<d$, \eqref{1.1} admits a semi-wavefront $(\omega, c^*(\sigma))$ with  $\max\{|c^*(\sigma)|,\|\omega\|_{C^2
(\mathbb R)}\}\leq M$, i.e. $\omega$ is the solution of the following problem
\begin{align}
	 & \omega''-c^*(\sigma)\omega'+\omega^2(1-J_\sigma*\omega)-d\omega=0 \qquad
	 \hbox{in}\; \mathbb R,\nonumber \\
	 & \omega(-\infty)=0 ,\quad \omega(0)=d_0,
\label{1.10}
\end{align}
and $0\leq\omega\leq A$ on $\mathbb R$, $\omega'\geq0$ in $(-\infty,0]$.

\item [(ii)] if furthermore $m_i=\int_\mathbb R|z|^iJ(z)dz<+\infty$ for $i=1,2$,
then there exists $\sigma^*>0$
such that for $\sigma<\sigma^*$, $c^*(\sigma)$ is positive and the semi-wavefronts are in fact
wavefronts with $\omega(+\infty)=A$.

\item [(iii)]  as $\sigma\to0$, $c^*(\sigma)$ converges to
$c^*$. Moreover, the solution $\omega$ has a subsequence converging to $\omega_0$ in $C^{1,\alpha}_{loc}(\mathbb R)$,
which satisfies
\begin{align}
	 & \omega_0''-c^*\omega_0'+\omega_0^2(1-\omega_0)-d\omega_0=0 \qquad
	 \hbox{in}\; \mathbb R\nonumber \\
	 & \omega_0(-\infty)=0,\quad \omega_0(0)=d_0,\quad \omega_0(+\infty)=A.
	 \label{24}
\end{align}
\end{enumerate}
\end{theorem}

Next we summarize the main methods used in this paper.
To study the existence of monotone traveling
wave, we use the classical method of sub- and super-solutions for an appropriate monotone operator,
which is motivated by \cite{22} on the time-delay Fisher-KPP
equation and \cite{6} on the nonlocal Fisher-KPP equation. In our case,
it's verified that the obtained monotone wavefronts connect the two positive states $a$ and $A$.
The proof of the existence of wavefronts connecting $0$ and $A$ is more delicate.
We start from a cut-off approximation, in a bounded domain $[-L,L]$, of the original problem and show that
the solutions are between $0$ and $A$.
Furthermore we can obtain the uniform $C^2$-bound of the solutions independent of $L$ and the scale of the cut-off.
By removing the cut-off and letting $L$ tend to infinity we derive the existence
of semi-wavefronts which connect $0$ to $d_0$. To show the semi-wavefronts are in fact wavefronts with
$\omega(+\infty)=A$, the main
difficulty is to exclude the case that $\omega(+\infty)=0$ and $\omega(+\infty)=a$.
Such a difficulty also arises in
the construction of bistable wavefronts in \cite{16,8}. Instead of using the energy methods as in \cite {16},
we adopt a rather direct method by comparing the semi-wavefronts that has been obtained from the nonlocal problem with those of the corresponding local problem.

The paper is organized as follows. In Section 2, by
monotone iteration method, we establish the existence of monotone wavefronts connecting the two
positive equilibrium $a$ and $A$. In Section 3, we prove the existence of semi-wavefronts by a
limiting process. Moreover, for $\sigma$ sufficiently small, we prove that the semi-wavefronts are wavefronts
connecting $0$ and $A$. Furthermore, as $\sigma\to0$, in both of Section 2 and Section 3, we prove that the
wavefronts converge to those of the corresponding local problems.

\section{Monotone wavefronts connecting $a$ and $A$}

To prove the existence of monotone wavefronts, we adopt the method of the sub- and super-solution.
The main task is to define a monotone operator and
to construct a pair of ordered lower and upper fixed points. To this end, we prove the following lemmata.
\begin{lemma}\label{lem2.1}
Denote
$$
F(\omega)(\xi)=2A\omega(\xi)-\omega^2(\xi)(1-J_\sigma*\omega(\xi)),
$$
then for any $0\leq\omega\leq A$, we have
\begin{enumerate}
\item[(i)] $F(\omega)(\xi)>0$;
\item[(ii)] $F(\omega_1)(\xi)\geq F(\omega_2)(\xi)\quad\hbox{if}\quad\omega_1(\xi)\geq\omega_2(\xi)$.
\end{enumerate}
\end{lemma}

\begin{proof}
\begin{enumerate}
\item[(i)]  It can be easily checked that
\begin{align*}
F(\omega)=&2A\omega-\omega^2(1-J_\sigma*\omega)
\\
=&\omega[2A-\omega(1-J_\sigma*\omega)]
\\
\geq&\omega(2A-A)>0
\end{align*}
\item[(ii)] Denote $g(\omega)=2A\omega-\omega^2$,
then $g'(\omega)=2A-2\omega\geq0$, which together with the monotonicity of $h(\omega)=\omega^2J_\sigma*\omega$
 with $\omega$ imply the monotonicity of $F(\omega)$ in $\omega$.
\end{enumerate}
\end{proof}

Note that if $\omega(\xi)$ satisfies \eqref{1.2}, then we have
$$\omega''(\xi)-c\omega'(\xi)+(2A-d)\omega(\xi)=F(\omega)(\xi).$$
Define
$$
L[\omega]=\omega''(\xi)-c\omega'(\xi)+(2A-d)\omega(\xi)-F(\omega)(\xi),
$$
then it is clear that finding a solution of \eqref{1.2} is equivalent to searching a function $\omega$ satisfying
$L[\omega]=0$, which is equivalent to
$$\omega(\xi)=\frac{1}{\mu_2-\mu_1}\int_\xi^{+\infty}\left(e^{\mu_1(\xi-y)}-e^{\mu_2(\xi-y)}\right)F(\omega(y))dy,$$
where $0<\mu_1\leq \mu_2$ are the two different real and positive roots of $\mu^2-c\mu+2A-d=0$ as
$c>2\sqrt{2A-d}$.

\begin{lemma}\label{lem2.2}
For $c>2\sqrt{2A-d}$, let
$$
T[\omega](\xi)=\frac{1}{\mu_2-\mu_1}\int_\xi^{+\infty}\left(e^{\mu_1(\xi-y)}-e^{\mu_2(\xi-y)}\right)
F(\omega(y))dy,
$$
then
\begin{enumerate}
\item[(i)] if $\overline{\omega}(\xi)$ is a super-solution of \eqref{1.2}, then $T[\overline{\omega}](\xi)\leq
\overline{\omega}(\xi)$ and $T[\overline{\omega}](\xi)$ is also a
super-solution. Moreover, for any sub-solution $\underline{\omega}(\xi)$ of \eqref{1.2} that satisfies
$\underline{\omega}(\xi)\leq\overline{\omega}(\xi)$, we have $\underline{\omega}(\xi)\leq
T[\overline{\omega}](\xi)$.
\item[(ii)] if $\omega(\xi)$ is increasing, then $T[\omega](\xi)$ is also increasing.
\end{enumerate}
\end{lemma}

\begin{proof}
\begin{enumerate}
\item[(i)] if $\overline{\omega}(\xi)$ is a super-solution of \eqref{1.2}, then
\begin{align}
L[\overline{\omega}]=\overline{\omega}''(\xi)-c\overline{\omega}'(\xi)+(2A-d)\overline{\omega}(\xi)-F(\overline{\omega})(\xi)
\geq0.\label{17}
\end{align}
 Let $\omega_1=T[\overline{\omega}]$, then
\begin{align}
\omega_1''(\xi)-c\omega_1'(\xi)+(2A-d)\omega_1(\xi)-F(\overline{\omega}(\xi))=0.\label{18}
\end{align}
Let $\varphi(\xi)=\overline{\omega}(\xi)-\omega_1(\xi)$, $r(\xi)=\varphi''(\xi)-c\varphi'(\xi)+(2A-d)\varphi(\xi)$,
then from
\eqref{17} and \eqref{18}, we obtain $r(\xi)\geq0$ and
$$\varphi(\xi)=\frac{1}{\mu_2-\mu_1}\int_\xi^{+\infty}\left(e^{\mu_1(\xi-y)}-e^{\mu_2(\xi-y)}\right)r(\xi)dy\geq0,$$
which means that $T[\overline{\omega}](\xi)\leq\overline{\omega}(\xi)$. Similarly, we can get $\underline{\omega}\leq
T[\overline{\omega}](\xi)$.

Furthermore, noticing $$\overline{\omega}\geq T[\overline{\omega}]=\omega_1,$$
from (ii) of Lemma 2.1 we derive
$$\omega_1''-c\omega_1'+(2A-d)\omega_1=F(\overline{\omega})\geq
F(\omega_1).$$
It follows that $L[\omega_1]\geq0$ and $\omega_1$ is also a super-solution.

\item[(ii)] If $\omega(\xi)$ is increasing, then from (ii) of Lemma 2.1 we obtain
$F(\omega)(\xi)$ is also increasing, therefore
$$
F(\omega)(\xi+t)-F(\omega)(\xi)\geq0, \quad \forall t>0.
$$
Furthermore, we have
\begin{align*}
&T[\omega](\xi+t)-T[\omega](\xi)
\\
=&\frac{1}{\mu_2-\mu_1}\int_\xi^{+\infty}\left(e^{\mu_1(\xi-y)}-e^{\mu_2(\xi-y)}\right)(F(\omega)(\xi+t)-
F(\omega)(\xi))dy\geq0,
\end{align*}
which implies that $T[\omega](\xi)$ is also increasing in $\xi$.
\end{enumerate}
\end{proof}

Define
$$\Phi_1(c,\sigma,\lambda):=\lambda^2-c\lambda-d-A^2\int_{\mathbb R}J_\sigma(s)e^{-\lambda s}ds=0$$
and
$$\Phi_2(c,\sigma,\lambda):=\lambda^2-c\lambda+d-A^2\int_{\mathbb R}J_\sigma(s)e^{-\lambda s}ds=0,$$
which, by a change of variable, are equivalent to
$$\frac{1}{\sigma^2}\lambda^2-\frac{c}{\sigma}\lambda-d-A^2\int_{\mathbb R}J(s)e^{-\lambda s}ds=0$$
and
$$\frac{1}{\sigma^2}\lambda^2-\frac{c}{\sigma}\lambda-d+A^2\int_{\mathbb R}J(s)e^{-\lambda s}ds=0$$
respectively. Similar to Proposition 2.1 in \cite{6}, we have the following result

\begin{proposition}\label{prop2.1}
For any $\sigma>0$, there exists $c_*(\sigma)\in(0,+\infty]$, which is increasing in $\sigma$, such that
if $c\geq c_*(\sigma)$, then $\Phi_i(c,\sigma,\lambda)=0$, $i=1,2$, admit the largest negative roots $\lambda_i$,
and there exists $\varepsilon_i=\varepsilon_i(c,\sigma)>0$ such that $\Phi_i(c,\sigma,\lambda_i-\varepsilon_i)>0$.
While if $c<c_*(\sigma)$, there exists $i\in\{1,2\}$ such that $\Phi_i(c,\sigma,\lambda)=0$ admits no negative root.
\end{proposition}

\begin{proof}
Due to the fact that for any fixed $\lambda<0$, $\frac{1}{\sigma^2}\lambda^2-\frac{c}{\sigma}\lambda$ is
increasing in $c\geq0$
and decreasing in $\sigma>0$, one has that for any $\sigma>0$, there exists $c_*(\sigma)\geq0$ such that
$c_*(\sigma)$ is increasing in $\sigma$ and $\Phi_i(c,\sigma,\lambda)=0$, $i=1,2$,
have at least one negative root if and only if $c\geq c_*(\sigma)$.
\end{proof}

Next, we will construct a pair of sub- and super-solutions in order to obtain a wavefront $\omega$.

For fixed $c>\max\{2\sqrt{2A-d},c_*(\sigma)\}$, let
\begin{align*}
\underline{\omega}(\xi)=\left\{
\begin{array}{clcc}\alpha e^{\mu\xi}+d,\hspace{0.4cm}&\;\xi\leq \xi_{-},\\
A(1-e^{\lambda_1\xi}),&\;\xi>\xi_{-},
\end{array}
\right.
\end{align*}
where $\mu>0$ is a solution of $\mu^2-c\mu+1=0$, $\lambda_1<0$ is a solution of $\Phi_1(c,\sigma,\lambda)=0$,
$\alpha$ and $\xi_{-}$ are uniquely determined by
\begin{align*}
\left\{\begin{array}{clcc}\alpha e^{\mu\xi_{-}}+d=A(1-e^{\lambda_1\xi_{-}}),\\
\alpha \mu e^{\mu\xi_{-}}=-\lambda_1 Ae^{\lambda_1\xi_{-}},
\end{array}\right.
\end{align*}
so that
$$\alpha e^{\mu\xi}+d\geq A(1-e^{\lambda_1\xi}), \;\forall \xi<\xi_{-}.$$

\begin{proposition}\label{prop2.2}
For $c\geq c_*(\sigma)$, $\underline{\omega}$ is a sub-solution of \eqref{1.2}, i.e., $L[\underline{\omega}]\leq0$.
\end{proposition}

\begin{proof}
For $\xi\leq\xi_{-}$, due to the fact that
$$d\leq\underline{\omega}\leq A\leq1,$$
we have
\begin{align*}
L[\underline{\omega}]&=\alpha(\mu^2-c\mu)e^{\mu\xi}
+\underline{\omega}^2(1-J_\sigma*\underline{\omega})-d\underline{\omega}
\\
&=(\mu^2-c\mu+1)(\underline{\omega}-d)-(\underline{\omega}-d)+\underline{\omega}^2
(1-J_\sigma*\underline{\omega})
-d\underline{\omega}
\\
&=-(1-\underline{\omega})(\underline{\omega}-d)-\underline{\omega}^2J_\sigma*\underline{\omega}<0.
\end{align*}

For $\xi>\xi_{-}$, noticing that
$$\alpha e^{\mu\xi}+d\geq A(1-e^{\lambda_1\xi}),$$
we have
\begin{align*}
L[\underline{\omega}]=&A(-\lambda_1^2+c\lambda_1)e^{\lambda_1\xi}+\underline{\omega}^2
(1-J_\sigma*\underline{\omega})-dA(1-e^{\lambda_1\xi})
\\
=&A(-\lambda_1^2+c\lambda_1+d)e^{\lambda_1\xi}
\\
&+\underline{\omega}^2
\left(1-\int_{-\infty}^{\xi_{-}}(\alpha e^{\mu s}+d)J_\sigma(\xi-s)ds
-\int_{\xi_-}^{+\infty}A(1-e^{\lambda_1 s})J_\sigma(\xi-s)ds\right)-dA
\\
\leq &-A^3\int_{\mathbb R}J_\sigma(s)e^{-\lambda_1(s-\xi)}ds+\underline{\omega}^2(1-\int_{\mathbb R}
A(1-e^{\lambda_1s})J_\sigma(\xi-s)ds)-A^2(1-A)
\\
= &-A^3\int_{\mathbb R}J_\sigma(s)e^{-\lambda_1(s-\xi)}ds+\underline{\omega}^2(1-A)+\underline{\omega}^2A
\int_{\mathbb R}e^{\lambda_1s}J_\sigma(\xi-s)ds)-A^2(1-A)
\\
= &A(\underline{\omega}^2-A^2)\int_{\mathbb R}J_\sigma(s)e^{-\lambda_1(s-\xi)}ds+(\underline{\omega}^2-A^2)(1-A)<0,
\end{align*}
where we have used that $d=A(1-A)$ and
$$\Phi_1(c,\sigma,\lambda_1)=\lambda_1^2-c\lambda_1-d-A^2\int_{\mathbb R}J_\sigma(s)e^{-\lambda_1 s}ds=0.$$
\end{proof}

Denote
\begin{align*}
\overline{\omega}(b,\xi)=\left\{\begin{array}{clcc}A(1-e^{\lambda_2\xi}+be^{(\lambda_2-\varepsilon_2)\xi}),
\hspace{0.4cm}&\;\xi\geq \xi_{b},\\
\mu_b,&\;\xi<\xi_{b},
\end{array}\right.
\end{align*}
where $\lambda_2<0$ is the largest negative root of $\Phi_2(c,\sigma,\lambda_2)=0$,
$\varepsilon_2>0$ is the constant such that
$\Phi_2(c,\sigma,\lambda_2-\varepsilon_2)>0,$
$b>0$ is a constant to be determined later, and $A(1-e^{\lambda_2\xi}+be^{(\lambda_2-\varepsilon_2)\xi})$ achieves its minimum $\mu_b$ at the point
$$
\xi=\xi_{b}=\frac1\varepsilon_2 \ln\frac{b(\lambda_2-\varepsilon_2)}{\lambda_2}
$$
with
$$
\mu_b=A(1-e^{\lambda_2\xi_b}+be^{(\lambda_2-\varepsilon_2)\xi_b})=A+\frac{\varepsilon_2 A}{\lambda_2-\varepsilon_2}
\left(\frac{b(\lambda_2-\varepsilon_2)}{\lambda_2}\right)^{\lambda_2/\varepsilon_2}.
$$
Since $\lambda_2<0$, it is easy to verify that for sufficiently large $b$, $\xi_b>0$ and $a<\mu_b<A$. Moreover,
 $\overline{\omega}$ is a $C^1$ function and is increasing with respect to $b>0$.

\begin{proposition}\label{prop2.3}
For $c\geq c_*(\sigma)$, $\overline{\omega}$ is a super-solution of \eqref{1.2} for $b\gg 1$,
i.e., $L[\overline{\omega}]\geq0$.
\end{proposition}

\begin{proof}
For $\xi<\xi_b$,
\begin{align*}
L[\overline{\omega}]&=\overline{\omega}^2(1-J_\sigma*\overline{\omega})-d\overline{\omega}
\\
&=\mu_b^2(1-A)+\mu_b^2(A-J_\sigma*\overline{\omega})-d\mu_b
\\
&=\mu_b^2(1-A)+\mu_b^2A\int_{\xi_b}^{+\infty}(e^{\lambda_2 s}-be^{(\lambda_2-\varepsilon_2)s})J_\sigma(\xi-s)ds
\\
&\qquad+\mu_b^2\int_{-\infty}^{\xi_b}(A-\mu_b)J_\sigma(\xi-s)ds-d\mu_b
\\
&\geq\mu_b^2(1-A)+\mu_b^2\int_{-\infty}^{\xi}(A-\mu_b)J_\sigma(\xi-s)ds-d\mu_b
\\
&= \mu_b^2(1-A)+\frac12\mu_b^2(A-\mu_b)-d\mu_b
\\
&=\mu_b(A-\mu_b)(\frac12\mu_b-(1-A)).
\end{align*}
For fixed $0\leq d<\frac29$, we have
$$A=\frac{1+\sqrt{1-4d}}{2}>\frac23.$$
For any
$$0<\varepsilon<\min\{1-\frac{9d}{2},3A-2\},$$
noticing that $\lim_{b\to\infty}\mu_b=A$,
we have
$$2(1-A)<A-\varepsilon<\mu_b<A$$
for  sufficiently large $b$ , which implies
$$L[\overline{\omega}]>0.$$

For $\xi\geq\xi_b$, noticing that
$$\Phi_2(c,\sigma,\lambda_2)=\lambda_2^2-c\lambda_2+d-A^2\int_{\mathbb R}J_\sigma(s)e^{-\lambda_2 s}ds=0$$
and
$$\Phi_2(c,\sigma,\lambda_2-\varepsilon_2)=(\lambda_2-\varepsilon_2)^2-c(\lambda_2-\varepsilon_2)
 +d-A^2\int_{\mathbb R}J_\sigma(s)e^{-(\lambda_2-\varepsilon_2) s}ds>0,$$
we have

\begin{align*}
L[\overline{\omega}]=&A(1-e^{\lambda_2\xi}+be^{(\lambda_2-\varepsilon_2)\xi})''-cA(1-e^{\lambda_2\xi}
+be^{(\lambda_2-\varepsilon_2)\xi})'+\overline{\omega}^2(1-J_\sigma*\overline{\omega})-d\overline{\omega}
\\
&\geq A(-\lambda_2^2+c\lambda_2)e^{\lambda_1\xi}+Ab([(\lambda_2-\varepsilon_2)^2-c(\lambda_2-\varepsilon_2)]
e^{(\lambda_2-\varepsilon_2)\xi}+\overline{\omega}^2(1-A)
\\
&\qquad+\overline{\omega}^2(A-J_\sigma*\overline{\omega})-dA(1-e^{\lambda_2\xi}
+be^{(\lambda_2-\varepsilon_2)\xi})
\\
&=A(-\lambda_2^2+c\lambda_2+d)e^{\lambda_2\xi}+Ab[(\lambda_2-\varepsilon_2)^2-c(\lambda_2-\varepsilon_2)-d]
e^{(\lambda_2-\varepsilon_2)\xi}+\overline{\omega}^2(1-A)
\\
&\qquad+\overline{\omega}^2A\int_{\mathbb R}(e^{\lambda_2(\xi-s)}-be^{(\lambda_2-\varepsilon_2)(\xi-s)})
J_\sigma(s)ds-dA
\\
&=\frac{1}{A}(A^2-\overline{\omega}^2)(-\lambda_2^2+c\lambda_2-d)e^{\lambda_2\xi}+
\frac{b}{A}(A^2-\overline{\omega}^2)[(\lambda_2-\varepsilon_2)^2-c(\lambda_2-\varepsilon_2)+d]
e^{(\lambda_2-\varepsilon_2)\xi}
\\
&\qquad+\frac{b}{A}\overline{\omega}^2\Phi_2(c,\sigma,\lambda_2-\varepsilon_2)e^{(\lambda_2-\varepsilon_2)\xi}
-(A^2-\overline{\omega}^2)(1-A)+2dA(e^{\lambda_2\xi}-be^{(\lambda_2-\varepsilon_2)\xi})
\\
&\geq\frac{b}{A}\overline{\omega}^2\Phi_2(c,\sigma,\lambda_2-\varepsilon_2)e^{(\lambda_2-\varepsilon_2)\xi}
-\frac{1}{A}(A^2-\overline{\omega}^2)(\lambda_2^2-c\lambda_2+d)e^{\lambda_2\xi}+(1-A)(A-\overline{\omega})^2
\\
&\geq\frac{b}{A}\overline{\omega}^2e^{(\lambda_2-\varepsilon_2)\xi}\left[\Phi_2(c,\sigma,\lambda_2-\varepsilon_2)
-\frac{A^2}{b\overline{\omega}^2}(2e^{(\lambda_2+\varepsilon_2)\xi}+2be^{2\lambda_2\xi})(\lambda_2^2-c\lambda_2+d)
\right]
\\
&\geq\frac{b}{A}\overline{\omega}^2e^{(\lambda_2-\varepsilon_2)\xi}\left[\Phi_2(c,\sigma,\lambda_2-\varepsilon_2)-
\frac{A^2}{b\overline{\omega}^2}(2e^{(\lambda_2+\varepsilon_2)\xi_b}+2be^{2\lambda_2\xi_b})(\lambda_2^2-c\lambda_2+d)
\right]
\\
&\geq\frac{b}{A}\overline{\omega}^2e^{(\lambda_2-\varepsilon_2)\xi}
\cdot\left[\Phi_2(c,\sigma,\lambda_2-\varepsilon_2)-\frac{2A^2}{a^2}\left(\frac1b\left(\frac{\lambda_2}{b(\lambda_2
-\varepsilon_2)}\right)^{-\lambda_2/\varepsilon_2-1}\right.\right.
\\
&\qquad\left.\left.+\left(\frac{\lambda_2}{b(\lambda_2-\varepsilon_2)}\right)
^{-2\lambda_2/\varepsilon_2}\right)
(\lambda_2^2-c\lambda_2+d)\right],
\end{align*}
where we have used the fact that $\xi_{b}=\frac1\varepsilon_2 \ln\frac{b(\lambda_2-\varepsilon_2)}{\lambda_2}$,
$a<\overline{\omega}<A$ and
$$(A^2-\overline{\omega}^2)\leq A^2(2e^{\lambda_2\xi}+2be^{(2\lambda_2-\varepsilon_2)\xi}).$$
For $b$ sufficiently large, since $\lambda_2<0$, it is easy to see that
$L[\overline{\omega}]>0$.
\end{proof}

\begin{lemma}\label{lem2.3}
Any solution $\omega\in C^2(\mathbb R)\cap L^\infty(\mathbb R)$ to \eqref{15} with
$\lim_{\xi\to-\infty}\omega(\xi)=\alpha_0$ and
$\lim_{\xi\to+\infty}\omega(\xi)=\beta_0$ has the property that $\alpha_0, \beta_0\in\{0,a,A\}$.
\end{lemma}

\begin{proof} Let $x_n\to\infty$, then the sequence of functions
 $v_n(x)=\omega(x+x_n)$ solve
$$
v_n''-cv_n'+v_n^2(1-J_\sigma*v_n)-dv_n=0, \quad \mbox{ in }\mathbb R.
$$
Since $\omega$ is bounded, $v_n$ is uniformly bounded with respect to $n$.
From the classical $W^{2,p}$ theory for second order linear elliptic equations, we obtain that for all
$1<p<\infty$,
$$\|v_n\|_{W_{loc}^{2,p}(\mathbb R)}\leq C.$$
From Sobolev embedding theorem, there is a subsequence of $v_n$, still denoted by $v_n$ itself,
 such that $v_n\to v$ strongly in
$C^{1,\alpha}_{loc}(\mathbb R)$ and weakly in $W^{2,p}_{loc}(\mathbb R).$ Then $v(x)\equiv\beta_0$
and
$$v''-cv'+v^2(1-J_\sigma*v)-dv=0,$$
which implies $\beta_0^2(1-\beta_0)-d\beta_0=0$ and $\beta_0\in\{0,a,A\}$. Similarly, we can prove that
$\alpha_0\in\{0,a,A\}$.
\end{proof}

\begin{proof}[Proof of Theorem 1.1] The proof consists of the following two parts.

\begin{enumerate}
\item[(i)] First we consider the case
$$
c>\max\{2\sqrt{2A-d},c_*(\sigma)\}.
$$
 Let $\overline{\omega}_0=\overline{\omega}$ and define the bounded continuous function sequence
  $\overline{\omega}_m$ by the
 following iteration scheme
$$\overline{\omega}_m''(\xi)-c\overline{\omega}_m'(\xi)+2\overline{\omega}_m(\xi)=F(\overline{\omega}_{m-1})(\xi).$$
Then from Lemma \ref{lem2.1} and Lemma \ref{lem2.2}, we can obtain that
for any $m$,
$$
\overline{\omega}_m(\xi)=T[\overline{\omega}_{m-1}](\xi)
$$
is increasing and satisfies
\begin{align}
\underline{\omega}\leq\cdots\leq\overline{\omega}_m\leq\cdots\leq\overline{\omega}_1\leq\overline{\omega}_0
=\overline{\omega}.\label{19}
\end{align}
Hence, there exists a increasing function $\omega(\xi)$ such that $\overline{\omega}_m(\xi)\to\omega(\xi)$ a.e. for $\xi\in\mathbb R$.
Therefore, we have
$$
\omega(\xi)=T[\omega](\xi),
$$
 which implies that $\omega$ is a solution of \eqref{15}.
Since $0\leq\omega(\xi)\leq A$ is increasing, there exist two non-negative constants $\alpha_0$, $\beta_0$ such
that
$$
\lim_{\xi\to-\infty}\omega(\xi)=\alpha_0,\quad\lim_{\xi\to+\infty}\omega(\xi)=\beta_0.
$$
By Lemma \ref{lem2.3},
we have $\alpha_0$, $\beta_0\in\{0,a,A\}$. Noticing that
$$
\lim_{\xi\to+\infty}\underline{\omega}(\xi)=\lim_{\xi\to+\infty}\overline{\omega}(\xi)=A,
$$
we have $\beta_0=A$. Furthermore,
$$
\lim_{\xi\to-\infty}\underline{\omega}(\xi)=d,\quad
\lim_{\xi\to+\infty}\overline{\omega}(\xi)=\mu_b<A,
$$
imply $d<\alpha_0<A$, then $\alpha_0=a$, which means that $\omega$ is a solution of \eqref{1.2}.

Since $a\leq\omega\leq A$ and $\omega'\geq0$, we claim that
$\omega'(\xi)\leq \mu_1A$ for $\xi\in\mathbb R$. In order to prove this, a direct computation from
$$
\omega(\xi)=\frac{1}{\mu_2-\mu_1}\int_\xi^{+\infty}\left(e^{\mu_1(\xi-y)}-e^{\mu_2(\xi-y)}\right)
F(\omega(y))dy
$$
gives that
$$
\omega'(\xi)=\frac{1}{\mu_2-\mu_1}\int_\xi^{+\infty}\left(\mu_1e^{\mu_1(\xi-y)}-\mu_2e^{\mu_2(\xi-y)}\right)
F(\omega(y))dy.
$$
Therefore,
$$
\omega'-\mu_1A\leq \omega'-\mu_1\omega=-\int_\xi^{+\infty}e^{\mu_2(\xi-y)}F(\omega(y))dy\leq0,
$$
by noticing that
$$
F(\omega)=2A\omega-\omega^2(1-J_\sigma*\omega)>0,
$$
and $0<\mu_1\leq\mu_2$ are the two positive roots of
$\mu^2-c\mu+2A-d=0$.
Furthermore,
$\|\omega\|_{C^2(\mathbb R)}\leq M$ can be obtained directly from \eqref{1.2}.

We are left to consider the case
$$
c=\max\{2\sqrt{2A-d}, c_*(\sigma)\}.
$$
Choosing $\{c_n\}$ such that
$c_n>\max\{2\sqrt{2A-d}, c_*(\sigma)\}$ and $c_n\to \max\{2\sqrt{2A-d}, c_*(\sigma)\}$, then for
each $n$, the above discussion gives a monotone travelling wavefront $\omega_n$ with speed $c_n$, such that
$$
\|\omega_n\|_{C^2(\mathbb R)}\leq M.
$$
By appropriate translations, we fix
$$
\omega_n(0)=\frac12, \quad \mbox{ for all } n.
$$
By
Arzel\`{a}-Ascoli theorem, $\omega_n$ and $\omega_n'$ have a locally uniformly convergent subsequence with limit
$\omega$, $\omega'$ and
$$
a\leq\omega\leq A\quad 0\leq\omega'\leq\mu_1A
$$
together with
$$
\omega(0)=\frac12,\quad\omega(-\infty)=a,\quad\omega(+\infty)=A.
$$

\item[(ii)] By Proposition \ref{prop2.1}, we have that $c_*(\sigma)\geq0$ and $c_*(\sigma)$ is decreasing as $\sigma\to0$.
Thus there exists $c_*\geq 0$ such that $c_*(\sigma)\to c_*$.
Next we take the limit $\sigma\rightarrow 0$.
Let $(\omega_\sigma,c)$ be the solution of \eqref{1.2} that has been obtained in the previous step,
where $c\geq \max\{2\sqrt{2A-d}, c_*(\sigma)\}$, and by appropriate translations, fix
$$
\omega_\sigma(0)=\frac12, \quad \mbox{ for all } \sigma,
$$
and
$$
\|\omega_\sigma\|_{C^2(\mathbb R)}\leq M.
$$
Therefore, $\omega_\sigma$ has a subsequence which converges to $\omega_0$ locally uniformly in $C^{1,\alpha}(\mathbb R)$
as $\sigma\to0$, where $\omega_0\in C^{2}(\mathbb R)$ is the solution of \eqref{21}, that is,
\begin{align*}
	 & \omega_0''-c\omega_0'+\omega_0^2(1-\omega_0)-d\omega_0=0 \qquad
	 \hbox{in}\; \mathbb R, \\
	 & \omega_0(-\infty)=a,\quad \omega_0(0)=1/2,\quad \omega_0(+\infty)=A.
\end{align*}
\end{enumerate}
\end{proof}

\section{Semi-wavefronts with $\omega(-\infty)=0$ and wavefronts connecting $0$ and $A$}

In this section, we study the existence of wavefronts connecting $0$ and $A$.

We construct the wavefronts connecting $0$ and $A$ by considering a
sequence of approximating problems on intervals
$[-L,L]$, and then pass to the limit $L\to\infty$. In particular, two difficulties arise in the proof. One comes from showing that the
speed $c$ and the $C^1$ norm of $\omega$ are controlled by a constant independent of $L$, and the other comes from establishing that the two equilibriums
$0$ and $A$ are
indeed reached at infinity.

For $L>0$, we introduce the homotopy parameter $0\leq\tau\leq 1$ and a smooth cut-off function
$g_\varepsilon(s)\in C_0^\infty(0,A)$ with $\varepsilon\in(0,A/6)$ such that $0\leq g_\varepsilon(s)\leq 1$ for $0\leq s\leq A$ and
$$
g_\varepsilon(s)\equiv 1 \quad \mbox{ for } s\in(3\varepsilon,A-3\varepsilon).
$$
We consider the following problem with cut-off both in space variable and in the nonlinear reaction,
\begin{align}
\omega''-c\omega'+\tau g_{\varepsilon}(\omega)[\omega^2(1-J_\sigma*\tilde{\omega})-d\omega]=0 \qquad
\hbox{in}\; (-L,L)\label{1.3}
\end{align}
with
\begin{align}
\omega(-L)=0,\quad \omega(L)=A,\label{1.4}
\end{align}
where $\tilde{\omega}$ is the extension of $\omega$ with $\omega=0$ on $(-\infty,-L)$ and $\omega=A$ on $(L,+\infty)$.

If
$$\max_{t\in[-L,L]}\omega(t)=\omega(t_0)>A \;\hbox{or}\;\min_{t\in[-L,L]}\omega(t)=\omega(t_0)<0,$$
then $t_0\in(-L,L)$ and $g_\varepsilon(\omega)=0$ in a neighborhood of $t_0$, which together with \eqref{1.3}
implies $\omega''-c\omega'=0$ in the same neighborhood.
The maximum principle implies that $\omega\equiv \omega(t_0)$, which is a contraction. Thus,
$$
0\leq \omega(t)\leq A\;\hbox{for all}\; t\in[-L,L].
$$

For fixed $d_0\in(0,d)$, we normalize the wavefront $\omega$ such that
\begin{align}
\max_{-L\leq t\leq0}\omega(t)=d_0.\label{1.5}
\end{align}
This constraint indirectly fixes the speed $c$.

We claim that $\omega$ is increasing in $[-L,0]$. In fact, if there exists a local maximal point
$t_0\in[-L,0)$ such that
$\omega''(t_0)\leq0$, $\omega'(t_0)=0$, then from \eqref{1.3}, we obtain
$$\omega(t_0)(1-J_\sigma*\omega(t_0))\geq d,$$
which contradicts to \eqref{1.5}.
Therefore, $\omega$ is increasing in $[-L,0]$ and $\omega(0)=d_0$.

\vskip5mm

The following lemma provides a priori bounds for $\|\omega\|_{C^2(-L,L)}$.

\begin{lemma}\label{lem3.1}
There exist $C$ and $L_0$ such that, for all $\tau\in[0,1]$,
$L\geq L_0$, $\varepsilon\in(0,\frac16A)$ and $\sigma>0$, any solution $(c, \omega)$ of
\eqref{1.3}-\eqref{1.5} satisfies
$$\|\omega\|_{C^2(-L,L)}\leq C.$$
\end{lemma}

\begin{proof} Noticing $0\leq\omega(t)\leq A$ for all $t\in[-L,L]$, which together with the fact that
$\omega(-L)=0$ and $\omega(L)=A$ imply that
$$
\omega'(-L)\geq0\;\hbox{ and }\;\omega'(L)\geq0.
$$
Let
$$H(t)=\tau g_{\varepsilon}(\omega)[\omega^2(1-J_\sigma*\omega)-d\omega],$$
then $|H(t)|\leq A^2$ for all $t\in[-L,L]$.
From
$$
(e^{-ct}\omega')'=-e^{-ct}H(t),
$$
we obtain that for $-L\leq t_1\leq t_2\leq L$,
\begin{align}
\omega'(t_1)e^{c(t_2-t_1)}+\frac{A^2}{c}(1-e^{c(t_2-t_1)})\leq\omega'(t_2)\leq \omega'(t_1)e^{c(t_2-t_1)}
-\frac{A^2}{c}(1-e^{c(t_2-t_1)})\label{03}
\end{align}
and
\begin{align}
\omega'(t_2)e^{c(t_1-t_2)}+\frac{A^2}{c}(e^{c(t_1-t_2)}-1)\leq\omega'(t_1)\leq \omega'(t_2)e^{c(t_1-t_2)}
-\frac{A^2}{c}(e^{c(t_1-t_2)}-1).\label{04}
\end{align}

We claim that
\begin{eqnarray}
\label{c>0} 0\leq\omega'(-L)\leq \frac1c A^2, &&\mbox{ for } c>0,\\
\label{c<0} 0\leq\omega'(L)\leq \frac1c A^2, &&\mbox{ for } c<0.
\end{eqnarray}
We first prove \eqref{c>0}.
Assuming the contrary, from \eqref{03}, by choosing $t_1=-L$,
we have
$$
\omega'(t_2)\geq\frac1c A^2+(\omega'(-L)-\frac1c A^2)e^{c(t_2-t_1)})\geq\frac1c A^2, \quad \mbox{ for all }t_2\in[-L,L].
$$
This cannot hold for a bounded function $0\leq\omega(t)\leq A$ and $\omega(L)=A$ for $L\geq L_0=C/(2A)$.
Similarly we can verify \eqref{c<0}.

\vskip5mm
Next we prove the boundedness of $\omega'(t)$ on $[-L,L]$ uniformly in $\tau$,
$L$, $\varepsilon$ and $\sigma$.

For $c>0$, with the change of variables
$$\omega(t)=e^{x(t)}-1,$$ we have
$$x'(t)=\frac{\omega'(t)}{\omega(t)+1}.$$
Then \eqref{1.3} is transformed into
$$x''-cx'+(x')^2+\tau g_{\varepsilon}(\omega)\frac{\omega}{\omega+1}\left[\omega
(1-J_\sigma*\omega)-d\right]=0.$$
Denote $y(t)=x'(t)$, we obtain
\begin{align}
y'-cy+y^2+f(t)=0,
\label{20}
\end{align}
where
 $$
 f(t)=\tau g_{\varepsilon}(\omega)\frac{\omega}{\omega+1}\left[\omega(1-J_\sigma*\omega)-d\right].
 $$
We have that $|f(t)|\leq A$, which is a direct consequence of $0\leq \omega(t)\leq A$.
$\omega\in C^2[-L,L]$ shows that $y(t)\in C^1[-L,L]$.
Let
$$
\beta=\min_{t\in[-L,L]}y(t),\quad\gamma=\max_{t\in[-L,L]}y(t).
$$
Next we will give a lower bound for $\beta$ and an upper bound for $\gamma$ uniformly in $\tau$, $\varepsilon$,
$L$ and $\sigma$.
Denote
$$\lambda_1(t)=\frac{c-\sqrt{c^2-4f(t)}}{2},\quad\lambda_2(t)=\frac{c+\sqrt{c^2-4f(t)}}{2},$$
which are the roots of $y^2-cy+f(t)=0$.
Suppose that $y(t)$ achieves its minimum at $t_1$, i.e.,
$$
\beta=\min_{t\in[-L,L]}y(t)=y(t_1).
$$
If $t_1=-L$, then
$$\beta=y(-L)=x'(-L)=\omega'(-L)\geq0.$$
If $t_1=L$, then
$$\beta=y(L)=x'(L)=\frac{\omega'(L)}{A+1}\geq0.$$
While if $t_1\in(-L,L)$, then $y'(t_1)=0$. From \eqref{20},
we obtain
$$y^2(t_1)-cy(t_1)+f(t_1)=0,$$
thus
$$\beta=y(t_1)\in\{\lambda_1(t_1),\lambda_2(t_1)\}\geq\frac{c-\sqrt{c^2+4A}}{2}.$$
On the other hand, suppose that $y(t)$ achieves its maximum at $t_2$, i.e.,
$$\gamma=\max_{t\in[-L,L]}y(t)=y(t_2).$$
If $t_2=-L$, then
$$\gamma=y(-L)=x'(-L)=\omega'(-L)\leq\frac{A^2}{c}.$$
If $t_2=L$, then $y'(L)\geq0$, from \eqref{20},
we obtain
$$y^2(L)-cy(L)+f(L)\leq 0,$$
thus
$$\gamma=y(L)\in(\lambda_1(L),\lambda_2(L))\leq\frac{c+\sqrt{c^2+4A}}{2}.$$
If $t_2\in(-L,L)$, then $y'(t_2)=0$. From \eqref{20},
we obtain
$$y^2(t_2)-cy(t_2)+f(t_2)=0,$$
then
$$\gamma=y(t_2)\in\{\lambda_1(t_2),\lambda_2(t_2)\}\leq\frac{c+\sqrt{c^2+4A}}{2}.$$
From the above discussion, we obtain that
$$\beta\geq\frac{c-\sqrt{c^2+4A}}{2}$$
and
$$\gamma\leq\max\{\frac{A^2}{c},\frac{c+\sqrt{c^2+4A}}{2}\}.$$
Furthermore, noticing $\omega(t)=e^{x(t)}-1$ and $\omega'(t)=(\omega(t)+1) y(t)$,
the uniform boundedness of $\omega'$ can be obtained.

For $c<0$, with the change of variables $\omega(t)=e^{-x(t)}-1$,
since
$$\frac{A^2}{c(A+1)}\leq x'(L)=-\frac{\omega'(L)}{A+1}\leq0$$
and
$$x'(-L)=-\omega'(-L)\leq0,$$
by similar analysis, the uniform boundedness of $\omega'$ achieves.

Now we have proved that the bounds of $\omega$ and $\omega'$ are independent of $\tau$,
$L$, $\varepsilon$ and $\sigma$. Then from \eqref{1.3}, for $c\neq0$,
the uniform boundedness of $\omega''$ can be obtained. While for the case $c=0$,
the uniform boundedness of $\omega''$ follows directly from \eqref{1.3}.
Finally, for any $c\in\mathbb R$, there exists a constant $C$ independent of
 $\tau$, $L$, $\varepsilon$ and $\sigma$ such that $\|\omega\|_{C^2(-L,L)}\leq C$.
\end{proof}

The next lemma provides an a priori bound for the speed $c$.

\begin{lemma}\label{lem3.2}
There exists $L_0>0$, for any $L>L_0$, there exists $K(L)>0$ such that for all $\tau\in[0,1]$,
$\varepsilon\in(0,\frac16A)$, any solution $(c,\omega)$ of
\eqref{1.3}-\eqref{1.5} satisfies
$-K(L)\leq c\leq c_{max}=2\sqrt{A}$. Moreover, for $\tau=1$, there exists $C>0$ such that for all
$L\geq L_0$, $\varepsilon>0$ and $\sigma>0$, we have
 $c\geq c_{min}=-C\sigma$.
\end{lemma}

\begin{proof}
Since $0\leq\omega\leq A\leq1$, the solution $\omega$ of \eqref{1.3} satisfies the inequality
\begin{eqnarray}\label{omegaA}
\omega''-c\omega'+A\omega\geq0.
\end{eqnarray}
We will prove $c\leq2\sqrt{A}$ for big enough $L$ by a contradiction argument. If $c>2\sqrt{A}$, let
$$h(t)=Ae^{\sqrt{A}(t-L)},$$
then
\begin{eqnarray}\label{eqnh}
ch'(t)>h''(t)+Ah(t).
\end{eqnarray}
Noticing
$$\omega(-L)=0<h(0),\quad \omega(L)=A=h(L),$$
by comparing the equations \eqref{omegaA} and \eqref{eqnh}, we have that $\omega(t)\leq h(t)$ in $(-L,L)$.
However,
$$d_0=\omega(0)\leq h(0)=Ae^{-L\sqrt{A}},$$
which is impossible
for
$$L>L_0=(\ln A-\ln d_0)/\sqrt{A}.$$
Hence, $c>2\sqrt{A}$ is impossible for $L$ sufficiently large.

Next we prove a lower bound for $c$ with given $L>0$. We consider a solution $(c, \omega)$ of \eqref{1.3}-\eqref{1.5}. It satisfies
$$\omega''-c\omega'\leq A^2\omega,$$
as well as $\omega(-L)=0$,  $\omega(L)=A$. If $v$ is the solution of $v''-cv'=dv$
 with $v(-L)=0$ and $v(L)=A$, then by comparison principle, we obtain $\omega\geq v$.
As $v$ can be computed explicitly and
$$v(0)=\frac{A}{e^{\lambda_+L}+e^{\lambda_-L}},\quad \lambda_{\pm}:=\frac{c\pm\sqrt{c^2+4d}}{2}.$$
We see that $v(0)\to1$ as $c\to-\infty$. It follows that, for any $L>0$, there exists $K(L)>0$
such that $c>-K(L)$ implies $v(0)>d_0$, which contradict with the fact that $\omega\leq v$ and $\omega(0)=d_0$.
Therefore, if $(c, \omega)$ is a solution of \eqref{1.3}-\eqref{1.5}, then $c\geq-K(L).$

In the end, we obtain a lower bound for the speed $c$ with $\tau=1$. Suppose that $c<-1$. We start by proving that
the derivative $\omega'$ is bounded by $-2A^2/c$ on an interval $[-L+K_0,L]$ with the constants $K_0$
independent of $L$. Choosing $t_1=-L$ in \eqref{04} and noticing that $\omega'(-L)\geq0$, we obtain
\begin{align}
\omega'(t_2)\geq \frac1c A^2\;\hbox{ for all }\; t_2\in[-L,L]\label{1.8}
\end{align}
and for some constant $K_0>-\frac{c}{A}$ independent of $L$, we have
\begin{align}
\omega'(t_2)\leq -\frac2c A^2\;\hbox{for any}\;t_2\in[-L+K_0,L].\label{1.6}
\end{align}
Otherwise $\omega'(t_1)\geq -\frac1c A^2e^{c(t_1-t_2)}$ for all $t_1\in[-L,t_2]$ which cannot hold for a bounded function with $0\leq\omega(t)\leq A$ and $\omega(-L)=0$ on interval $[-L,t_2]$ for big enough $L$.

For a fixed $\varepsilon_0=\frac{1-4d}{36},$ there exists
$R_0>0$ independent of $\sigma$ such that for $R=R_0\sigma$,
$$
A\int_{[-R,R]^c}J_\sigma(x)dx=A\int_{[-R_0,R_0]^c}J(x)dx\leq\varepsilon_0.
$$
We are going to prove that
$$
c\geq c_{min}=-\frac{2}{\varepsilon_0}A^2R_0\sigma.
$$
If this is not true, assume $c\leq c_{min}$.
Thanks to the
conditions $\omega(L)=A$ and $\omega(0)=d_0<d<\frac29$,
we can define $t_0>0$ as the smallest positive real such that $\omega(t_0)=\frac12$. From \eqref{1.6}we obtain
for $t\in[t_0-R,t_0+R]\cap[-L,L]$, we have
\begin{align}
d_0<\frac12-\varepsilon_0\leq\frac12+\frac{2A^2R}{c}\leq\omega(t)\leq \frac12-\frac{2A^2R}{c}\leq\frac12
+\varepsilon_0<A\label{25}
\end{align}
and $[t_0-R,t_0+R]\subset[0,L]$ as soon as $c\leq-\frac{2A^2R}{\varepsilon_0}=c_{min}$. Furthermore, we have
\begin{align}
J_\sigma*\omega(t)=&\int_{[-R,R]}J_\sigma(x)\omega(t-x)dx+\int_{[-R,R]^c}J_\sigma(x)\omega(t-x)dx\nonumber
\\
&\leq \frac12+\varepsilon_0+A\int_{[-R,R]^c}J_\sigma(x)dx\leq\frac12+2\varepsilon_0\label{24new}
\end{align}
as soon as $c\leq c_{min}$.

For  $c\leq c_{min}$,  $\omega$ is increasing on $(t_0,t_0+R)$. If not,
the definition of $t_0$ implies the existence of a local minimum
$\overline{t}\in(t_0,t_0+R)$. Noticing $d<\frac29$, $\omega'(\overline{t})=0$,
$\omega''(\overline{t})\geq0$, from \eqref{1.3}, we have $J_\sigma*\omega(\overline{t})\geq
1-d/\omega(\overline{t})$, which together with \eqref{25} and \eqref{24new} implies
$$\frac12+2\varepsilon_0\geq J_\sigma*\omega(\overline{t})\geq
1-d/\omega(\overline{t})\geq1-\frac{2d}{1-2\varepsilon_0}>\frac12+2\varepsilon_0,$$
 which is a contraction.

Therefore, for $c\leq c_{min}$, $t\in(t_0,t_0+R)$, we have $\omega'(t)\geq0$ and thus
\begin{align*}
\omega''&\leq \omega''-c\omega'=g_\varepsilon(\omega)[d\omega-\omega^2(1-J_\sigma*\omega)]
\\
&\leq\omega(d+\omega(2\varepsilon_0-\frac12))
\\
&\leq d-(\frac12-2\varepsilon_0)(\frac12-\varepsilon_0)
\\
&<d-\frac14+\frac32\varepsilon_0=-\frac{13}{2}\varepsilon_0.
\end{align*}
It follows that $\omega'(t_0)-\omega'(t_0+R)\geq\frac{13}{2}\varepsilon_0R$, which together with \eqref{1.8} and
\eqref{1.6} implies $c\geq-\frac{6A^2R}{13\varepsilon_0}>c_{min}$, which is a contraction.

Finally, it is proved that $c_{min}=-\frac{2}{\varepsilon_0}A^2R_0\sigma$ is an explicit
lower bound for $c$.
\end{proof}

Now we begin the homotopy argument. The a priori bounds obtained in Lemma \ref{lem3.1} and \ref{lem3.2}
allow us to use the
Leray-Schauder topological degree argument to prove existence of solutions to the problem \eqref{1.3}-\eqref{1.5}
with $\tau=1$ on the bounded domain $[-L,L]$.

\begin{proposition}\label{prop3.1}
There exist $K>0$ and $L_0$ such that, for all $L\geq L_0$,  $\varepsilon\in(0,A/6)$ and $\sigma>0$,
\eqref{1.3}-\eqref{1.5} with $\tau=1$ has a solution $(c, \omega)$, i.e., $(c, \omega)$ satisfies
\begin{align}
\left\{\begin{array}{clcc}
&\omega''-c\omega'+g_{\varepsilon}(\omega)[\omega^2(1-J_\sigma*\tilde{\omega})-d\omega]=0 \qquad \hbox{in}\; (-L,L),
\\
&\omega(-L)=0,\quad \omega(0)=d_0,\quad \omega(L)=A
\end{array}\right.\label{1.9}
\end{align}
with
$$
\|\omega\|_{C^2(-L,L)}\leq K, \quad -c_{min}\leq c\leq c_{max}.
$$
\end{proposition}

\begin{proof}
 We introduce a map $K_\tau$ which is defined from the Banach space $X=\mathbb R\times C^{1}[-L,L]$, equipped with the norm
$\|(c,v)\|_{X}=\max\{|c|,\|v\|_{C^{1}[-L,L]}\}$, onto itself, i.e.,
$$
K_\tau:(c,v)\to(d_0-v(0)+c,\omega),
$$
where $\omega$ is the solution of the linear system
\begin{align}
P_\tau\left\{\begin{array}{clcc}
&\omega''-c\omega'+\tau g_{\varepsilon}(v)[v^2(1-J_\sigma*\tilde{v})-dv]=0 \qquad \hbox{in}\; (-L,L),
\\
&\omega(-L)=0,\quad \omega(L)=A.
\end{array}\right.
\end{align}
A solution $(c_\tau,\omega_\tau)$ of the finite interval problem \eqref{1.3}-\eqref{1.5} is a fixed point
 of $K_\tau$ and satisfies
$K_\tau(c_\tau,\omega_\tau)=(c_\tau,\omega_\tau)$ and vice versa.
Hence, in order to show that \eqref{1.9} has a wavefront, it suffices to show that the kernel
of the operator $Id-K_1$ is nontrivial.  The classical regularity
theory implies that the operator $K_\tau$ is compact and continuous in $\tau\in[0,1]$.
Let
$B_M=\{\|(c,v)\|_X\leq M\}$. Then Lemma \ref{lem3.1} and \ref{lem3.2} show that the operator $Id-K_\tau$ does not vanish on the boundary
$\partial B_M$ with $M$ sufficiently large for any $\tau\in[0,1]$. It remains only to show that
$deg(Id-K_1,B_M,0)\neq0$ in $\overline{B}_M$. The homotopy invariance property of the degree implies that
$deg(Id-K_1,B_M,0)=deg(Id-K_0,B_M,0)$. Moreover, for $\tau=0$, the operator $F_0=Id-K_0$ is given by
$$F_0(c,v)=(v(0)-d_0,v-\omega_0^c).$$
Here $\omega_0^c(t)$ solves
\begin{align*}
&(\omega_0^c)''-c(\omega_0^c)'=0,
\\
& \omega_0^c(-L)=0,\quad,\omega_0^c(L)=A
\end{align*}
and is given by
\begin{align*}
\omega_0^c(t)=\left\{\begin{array}{clcc}
&A\frac{e^{ct}-e^{-cL}}{e^{cL}-e^{-cL}},\quad &&c\neq0,
\\
&\frac{1}{2L}t+1/2A,\quad&&c=0.
\end{array}\right.
\end{align*}
In particular,  since $\omega_0^c(0)$ is decreasing in $c$, $\omega_0^0(0)=\frac A2>\frac d2$ and
$\lim_{c\to+\infty}\omega_0^c(0)=0$,
there exists a unique $c_0$ such
that $\omega_0^{c_0}(0)=d_0$. The mapping $F_0=Id-K_0$ is homotopic to
 $$\Phi(c,v)=(\omega_0^c(0)-d_0,v-\omega_0^{c_0}).$$
The degree of the mapping $\Phi$ is the product
of the degrees of each component. As $\omega_0^c(0)$ is decreasing in $c$, $deg(\omega_0^c(0)-d_0,B_M,0)=-1$.
While $deg(v-\omega_0^{c_0},B_M,0)=1$.
Thus
$$
deg(Id-K_1,B_M,0)=deg(Id-K_0,B_M,0)=-1,
$$
and thereafter a solution  $(\omega,c)\in B_M$ of $P_1$ exists.
\end{proof}

The following lemma is used as a preparation in passing to the limit $L\rightarrow \infty$ and $\varepsilon\rightarrow 0$.

\begin{lemma}\label{lem3.3}
For any solution $(c, \omega)$ of \eqref{1.10} with $\omega\in C^2(\mathbb R)$ and
\begin{align}
|c|>\sqrt{m_2}\sigma A^2,
\end{align}
where $m_i=
\int_\mathbb R|z|^iJ(z)dz$, $i=1,2$, it holds that
$\lim_{t\to+\infty}\omega\in\{0,a,A\}$.
\end{lemma}

\begin{proof}
 Rewrite the first equation of \eqref{1.10} as
$$\omega''-c\omega'+\omega^2(1-\omega)+\omega^2(\omega-J_\sigma*\omega)-d\omega=0,$$
then multiply it by $\omega'$ and integrate from $-R$ to $R$ for arbitrary $R$, we get
\begin{align}
c\int_{-R}^{R}|\omega'|^2dt=[\frac12(\omega')^2+\frac{1}{3}\omega^3-\frac{1}{4}\omega^4-\frac d2\omega^2]|_{-R}^{R}
+\int_{-R}^R\omega'\omega^2(\omega-J_\sigma*\omega)dt.\label{1}
\end{align}
Denote the last term by $I$, Cauchy's inequality implies
\begin{align}
I^2\leq \int_{-R}^{R}(\omega'\omega^2)^2dt\int_{-R}^{R}(\omega-J_\sigma*\omega)^2dt\leq
 A^4\int_{-R}^{R}(\omega')^2dt
\int_{-R}^{R}(\omega-J_\sigma*\omega)^2dt.\label{2}
\end{align}
Noticing
$$\omega(t)-J_\sigma*\omega(t)=\int_{\mathbb R}J_\sigma(t-y)(\omega(t)-\omega(y))dy=
\int_{\mathbb R}\int_0^1J_\sigma(-z)
\omega'(t+\theta z))(-z)d\theta dz,$$
again by Cauchy's inequality we obtain
$$(\omega(t)-J_\sigma*\omega(t))^2\leq m_2\sigma^2\int_\mathbb R\int_0^1J_\sigma(-z)\omega'^2(t+\theta z))d\theta dz.$$
Integrating the above inequality, we have
\begin{align}
\int_{-R}^{R}(\omega(t)-J_\sigma*\omega(t))^2dt&\leq m_2\sigma^2\int_0^1\int_{\mathbb R}J_\sigma(-z)
\int_{-R+\theta z}^{R+\theta z}
\omega'^2(t)dtdzd\theta\nonumber
\\
&\leq m_2\sigma^2\int_{-R}^{R}\omega'(t)^2dt+m_2\sigma^2\int_0^1\int_{\mathbb R}J_\sigma(-z)2\theta|z|C^2dzd\theta.\label{3}
\\
&\leq  m_2\sigma^2\int_{-R}^{R}\omega'(t)^2dt+m_2m_1\sigma^3 C^2.\nonumber
\end{align}
A combination of \eqref{1}, \eqref{2} and \eqref{3} gives us
\begin{align}
|c|\int_{-R}^{R}\omega'(t)^2dt&\leq\left[\frac12(\omega')^2+\frac{1}{3}\omega^3-\frac{1}{4}\omega^4
-\frac d2\omega^2\right]\big|_{-R}^{R}\nonumber
\\
&\quad+A^2\sqrt{\int_{-R}^{R}\omega'(t)^2dt\left(m_2\sigma^2\int_{-R}^{R}\omega'(t)^2dt
+m_1m_2\sigma^3C^2\right)}\nonumber
\\
&\leq C+A^2\sqrt{m_2}\sigma\left(\int_{-R}^{R}\omega'(t)^2dt+m_1\sigma C^2\right).
\end{align}
If $|c|>A^2\sqrt{m_2}\sigma$, then
$\omega'\in L^2(\mathbb R)$. Furthermore, $\omega\in C^2(\mathbb R)$ implies that $\lim_{t\to\pm\infty}\omega'=0$, thus $\lim_{t\to+\infty}\omega$ exists.
Then from Lemma \ref{lem2.3}, we have $\lim_{t\to+\infty}\omega\in\{0,a,A\}$.
\end{proof}

\begin{proof}[Proof of Theorem 1.2]
From Proposition \ref{lem3.1}, for each $L\gg 1$ and $0<\varepsilon\ll 1$,
the problem \eqref{1.9}
does have at least one solution $(c_{L,\varepsilon},\omega_{L,\varepsilon})$. Next we will show that, for
$\varepsilon\to0$ and then
$L$ going to $+\infty$, the sequence $(c_{L,\varepsilon},\omega_{L,\varepsilon})$(or an extracted subsequence)
converges to a solution of \eqref{1.10}.

\begin{enumerate}
\item[(i)] Having constructed a solution $(c_{L,\varepsilon},\omega_{L,\varepsilon})$ of
\eqref{1.9}
with
$$-c_{min}\leq c_{L,\varepsilon}\leq c_{max},\quad \|\omega_{L,\varepsilon}\|_{C^2(-L,L)}\leq K$$
and noticing that $K$, $c_{min}$ and $c_{max}$
are uniform in $L\geq L_0$ and $\varepsilon\in(0,A/6)$. We can take the limit $\varepsilon\to0$
and $L\to+\infty$ in the approximating problem, and show that the limit $(\omega,c)$ is the wavefront that connects
$0$ and $A$. Namely, with fixed $L$, for $\varepsilon\to0$, there exists a subsequence of
$c_{L,\varepsilon}$ and $\omega_{L,\varepsilon}$, denoted by itself, such that
$c_{L,\varepsilon}\to c_L$ and $\omega_{L,\varepsilon}\to\omega_L$ in $C^1_{loc}(\mathbb R)$.
Then
$$-c_{min}\leq c_L\leq c_{max},\quad\|\omega_L\|_{C^2(-L,L)}\leq K.$$
Moreover,
from the definition of $g_\varepsilon$, we have $g_{\varepsilon}\to 1$
 in $\mathbb R$ as $\varepsilon\rightarrow 0$. Then $(c_L,\omega_L)$ is the solution of
\begin{align*}
\left\{\begin{array}{clcc}
\omega''-c\omega'+\omega^2(1-J_\sigma*\tilde{\omega})-d\omega=0 \qquad \hbox{in}\; (-L,L),
\\
\omega(-L)=0,\quad \omega(0)=d_0,\quad \omega(L)=A.
\end{array}\right.
\end{align*}

Again, there exists a subsequence $L_n\to\infty$, such that
$c_{L_n}\to c^*(\sigma)$ and $\omega_{L_n}\to \omega$ in $C^1_{loc}(\mathbb R)$,
and
$$
-c_{min}\leq c^*(\sigma)\leq c_{max},\quad \|\omega\|_{C^2(\mathbb R)}\leq M,
$$
together with
$$
0\leq\omega\leq A, \mbox{ in } \mathbb R, \mbox{ and } 0\leq\omega\leq d_0 \mbox{ in } (-\infty,0].
$$
Furthermore, the limit $(c^*(\sigma),\omega)$ is a solution of $\eqref{1.10}$, with $\omega'(t)\geq0$
for $t\in(-\infty,0]$. Hence, as $t\to-\infty$,
$\omega(t)\to \alpha_0$ and $0\leq \alpha_0\leq d_0$. By Lemma \ref{lem2.3}, we have $\alpha_0=0$.

\item[(ii)] In order to show that $\omega(+\infty)=A$, we have to start from the approximation $\omega_L(t)$
 in $[-L,L]$ with $\omega_L(L)=A$ and then take the limit $L\to +\infty$. In other words, we need to prove $\lim_{L\to+\infty}\omega_L(L)=A$.
The uniform bound of $|\omega_L'|_{\infty}\leq C$ provides that
\begin{align*}
|\omega_L-J_\sigma*\omega_L|(t)\leq&\int_{\mathbb R}J_\sigma(s)|\omega_L(t)-\omega_L(t-s)|ds
\\
\leq&\|\omega_L'\|_\infty \int_{\mathbb R} sJ_\sigma(s)ds\
\\
\leq&Cm_1\sigma.
\end{align*}
Thus
\begin{align*}
\omega_L''-c_L\omega_L'+\omega_L^2(1-\omega_L)-d\omega_L=\omega^2(J_\sigma*\omega_L-\omega_L)
\leq A^2Cm_1\sigma.
\end{align*}
Denote $f(s)=s^2(1-s)-ds$, the equation can be rewritten as
$$
\omega_L''-c_L\omega_L'+f(\omega_L)- C_0\sigma\leq 0,
$$
where $C_0=A^2Cm_1$. Let $\alpha<\gamma<\beta$ be the three solutions of $f(s)- C_0\sigma=0$.  There exists
$\sigma_0>0$ such that that for $\sigma<\sigma_0$, it holds $\alpha<0<a<\gamma<\beta<A$.
Let $\psi_L$ be the solution of
\begin{align}
&\psi_L''-c_L\psi_L'+ f(\psi_L)- C_0\sigma=0\;\hbox{ in }[-L,L],\label{1.11}
\\
&\psi_L(-L)=\alpha,\psi_L(L)=\beta.\label{1.12}
\end{align}
By maximum principle, we have that $\alpha\leq\psi_L\leq\beta$.
By comparison principle as in \cite{16}, we get that $\psi_L(t)\leq \omega_L(t)$ for $t\in[-L,L]$.
Then by the classical theory of elliptic equations, there exists a subsequence
${L_n}$, denoted by itself, such that as $n\to\infty$, $L_n\to\infty$,
$c_{L_n}\to c^*(\sigma)$ and $\psi_{L_n}\to\psi$ in $C^1_{loc}(\mathbb R)$.
and the limit satisfies the following local problem
\begin{align}
\psi''-c^*(\sigma)\psi'+ f(\psi)- C_0\sigma=0\;\hbox{ in }\;\mathbb R.\label{1.13}
\end{align}
It can be easily verified that $\psi\leq\omega$ in $\mathbb R$.

Now we claim that there exists a constant $C$ such that
$|\psi'(t)|\leq C$, for any $t\in\mathbb R$. This can be proved in the following. We reformulate \eqref{1.13} into its integral version
$$
\psi-\alpha=\frac{1}{z_2-z_1}\left(\int_{-\infty}^t e^{z_1(t-s)}r(\psi)(s)ds+\int_t^{+\infty}
e^{z_2(t-s)}r(\psi)(s)ds\right),
$$
where
$$
r(\psi)=\psi^2(1-\psi)-d\psi+b(\psi-\alpha)+C_0\sigma>0, \mbox{ for } b\gg 1,
$$
and
$$
z_1=\frac{c^*(\sigma)-\sqrt{(c^*(\sigma))^2+4b}}{2}<0,\quad z_2=\frac{c^*(\sigma)+\sqrt{(c^*(\sigma))^2+4b}}{2}>0.
$$
The first order derivative of $\psi$ is
$$
\psi'=\frac{1}{z_2-z_1}\left(z_1\int_{-\infty}^t e^{z_1(t-s)}r(\psi)(s)ds+z_2\int_t^{+\infty}
e^{z_2(t-s)}r(\psi)(s)ds\right).
$$
Thus we have
$$\psi'-z_1(\psi-\alpha)=\int_t^{+\infty}e^{z_2(t-s)}r(\psi)(s)ds\geq0$$
and
$$\psi'-z_2(\psi-\alpha)=-\int_{-\infty}^t e^{z_1(t-s)}r(\psi)(s)ds\leq0,$$
which imply that
$$z_1(\psi-\alpha)\leq\psi'\leq z_2(\psi-\alpha).$$
Then the boundedness of $\psi'$ follows from the boundedness of $\psi$. For any $R>0$,
multipling \eqref{1.13} by $\psi$
and integrating from $-R$ to $R$, we get
\begin{align}
c^*(\sigma)\int_{-R}^R|\psi'|^2dt=[\frac12(\psi')^2+\frac13\psi^3-\frac14\psi^4
-\frac{d}{2}\psi^2-C_0\sigma\psi]
 \big|_{-R}^R\leq C.\label{22}
 \end{align}
Next, we need to prove that $c^*(\sigma)$ is strictly positive in order to show that $\psi'$  is bounded in $L^2$.

Noticing that
$$
\beta^3=\beta^2-d\beta-C_0\sigma,\quad\alpha^3=\alpha^2-d\alpha-C_0\sigma,
$$
we have that
$$
\alpha\to0, \beta\to A \mbox{ as } \sigma\to0.
$$
Together with the fact that,
$$
\left(\frac13-\frac14A^2-\frac12 d\right)A-\frac13 d>0, \quad\mbox{ for } d<\frac29,
$$
it is easy to verify that there exists $\sigma_1>0$ such that for $\sigma<\sigma_1$ and big enough $L$,
\begin{eqnarray}
\nonumber &&c_L\int_{-L}^{L}|\psi_L'|^2dt\\
&=&\Big[\frac12(\psi_L')^2+\frac13\psi_L^3-\frac14\psi_L^4
-\frac{d}{2}\psi_L^2-C\sigma\psi_L\Big]\Big|_{-R}^R\nonumber
 \\
 &\geq&\Big[\frac13\psi_L^3-\frac14\psi_L^4-\frac{d}{2}\psi_L^2-C\sigma\psi_L\Big]\Big|_{-R}^R\label{23}
\\
&=&\frac13(\beta^3-\alpha^3)-\frac14(\beta^4-\alpha^4)-\frac{d}{2}(\beta^2-\alpha^2)-C\sigma(\beta-\alpha)\nonumber
\\
&=&\frac13(\beta^2-\alpha^2-d(\beta-\alpha))-(\frac14(\beta^2+\alpha^2)+\frac d2)(\beta^2-\alpha^2)-C\sigma(\beta-\alpha)\nonumber
\\
&=&\left(\left[\frac13-\frac14(\beta^2+\alpha^2)-\frac12 d\right](\beta+\alpha)-\frac13 d-C\sigma\right)(\beta-\alpha)>0.\nonumber
\end{eqnarray}
From \eqref{22} and \eqref{23}, we obtain
$$0<c^*(\sigma)\int_{\mathbb R}|\psi'|^2dt\leq C.$$
Thus
$$c^*(\sigma)>0, \quad 0<\int_{\mathbb R}|\psi'|^2dt<\frac{C}{c^*(\sigma)},$$
which implies that
$\psi'\in L^2$.
By the same arguments as that in Lemma \ref{lem2.3}, we obtain $\lim_{t\to\pm\infty}\psi(t)$ exists and belong to $\{\alpha,\gamma,\beta\}$.
Furthermore, noticing $\psi\leq\omega$ and $\omega(-\infty)=0$, we have $\psi(-\infty)=\alpha$. Now we claim that
$\psi(+\infty)>\alpha$. If $\psi(+\infty)=\alpha$, noticing that $c^*(\sigma)>0$, then from the theory of travelling waves for local problem,
there must holds $\psi\equiv\alpha$, which contradict with the fact that  $\int_{\mathbb R}|\psi'|^2dt>0$.
Thus $\psi(+\infty)>\alpha$,
which means that either $\psi(+\infty)=\beta$ or $\psi(+\infty)=\gamma$. Again from the theory of travelling waves for local
problem,
we obtain there exists $c_0$ independent of $\sigma$ such that $c^*(\sigma)>c_0>0$.
Then Lemma \ref{lem3.3} implies that for
$$\sigma<\sigma^*=\{\sigma_0,\sigma_1, \frac{c_0}{\sqrt{m_2}A^2}\},$$  $\lim_{t\to+\infty} \omega$
exists and belongs to $\{0,a,A\}$.
Noticing that $\beta>\gamma>a$, we have
$$
\omega(+\infty)\geq \psi(+\infty)>a.
$$
Therefore,
$\omega(+\infty)=A$.

\item[(iii)] In the end, as a byproduct, we can also take the limit $\sigma\to0$. Let
$\omega_\sigma$ be the solution of \eqref{1.10} with
$\omega_\sigma(+\infty)=A$. Noticing that
$$\|\omega_\sigma\|_{C^2(\mathbb R)}\leq K,\quad -c_{min}\leq c^*(\sigma)\leq c_{max}$$ with $K$,
$c_{min}$ and $c_{max}$ independent of $\sigma$ for $\sigma<\sigma_0$, we get a subsequence of
$(c^*(\sigma),\omega_\sigma)$, denoted by itself, such that
 $c^*(\sigma)\to c^*$ and $\omega_\sigma\to \omega_0$ locally uniformly in $C^{1,\alpha}(\mathbb R)$ as
 $\sigma\to0$,
 where $(c^*,\omega_0)$ is the solution of \eqref{24}, i.e.,
\begin{align*}
	 & \omega_0''-c^*\omega_0'+\omega_0^2(1-\omega_0)-d\omega_0=0 \qquad
	 \hbox{in}\; \mathbb R\\
	 & \omega_0(-\infty)=0,\quad \omega_0(0)=d_0,\quad \omega_0(+\infty)=A.
\end{align*}

\end{enumerate}
\end{proof}


\begin{thebibliography}{99}
\bibitem{7} M.~Alfaro and J.~Coville, Rapid traveling waves in the nonlocal Fisher equation connect two
unstable states, \it Appl. Math. Lett., \bf 25 \rm(2012), 2095-2099.

\bibitem{8} M.~Alfaro, J.~Coville, G.~Raoul, Bistable travelling waves for nonlocal reaction diffusion
equations, \it Discrete Contin. Dyn. Syst. Ser. A. \bf 34 \rm(2014), 1775-1791.

\bibitem{3} N.~Apreutesei, A.~Ducrot and V.~A.~Volpert, Wavefronts for integro-differential equations in
population dynamics,
\it Discrete Cont. Dyn. Syst. Ser. B, \bf 11\rm(3) (2009), 541-561.

\bibitem{10} D.~G.~Aronson and H.~F.~Weinberger, Nonlinear diffusion in population genetics, com-
bustion, and nerve pulse propagation, \it Partial differential equations and related topics\rm
(Program, Tulane Univ., New Orleans, La., 1974), 5-49. \it Lecture Notes in Math., \bf
446, \rm Springer, Berlin, 1975.

\bibitem{11} D.~G.~Aronson and H.~F.~Weinberger, Multidimensional nonlinear diffusion arising in
population genetics, \it Adv. in Math. \bf 30\rm(1)(1978), 33-76.


\bibitem{16} H. Berestycki and L. Nirenberg, Travelling fronts in cylinders, \it Ann. Inst. H. Poincar\'e
Anal. Non Lin\'eaire \bf 9 \rm(5)(1992), 497-572.

\bibitem{4} H.~Berestycki, G.~Nadin, B.~Perthame and L. ~Ryzhik, The non-local Fisher-KPP equation:
Wavefronts and steady states, \it Nonlinearity, \bf 22 \rm(2009), 2813-2844.


\bibitem{12} H.~Berestycki, B.~Nicolaenko and B.~Scheurer, Traveling wave solutions to combustion
models and their singular limits, \it SIAM J. Math. Anal. \bf 16\rm(6)(1985), 1207-1242.


\bibitem{BCL} S. Bian, L. Chen and E. Latos, Global existence and asymptotic behavior of solutions to a nonlocal Fisher-KPP type problem. {\it Nonlinear Anal.} {\bf 149}, (2017), 165-176.


\bibitem{BC} S. Bian and L. Chen, A nonlocal reaction diffusion equation and its relation with Fujita exponent,
{\it Journal of Mathematical Analysis and Applications,} {\bf 444}, (2016) 1479-1489.

\bibitem{CLP} L. Chen, E. Latos and G. Psaradakis, Blow-up criterion for a non-local Fisher-KPP type problem,
{\it in preparation.}


\bibitem{6} J.~Fang and X.-Q.~Zhao, Monotone wave fronts of the nonlocal Fisher-KPP equation, \it
Nonlinearity, \bf 24 \rm(2011), 3043-3054.

\bibitem{13} P.~C.~Fife and J.~B.~McLeod, The approach of solutions of nonlinear diffusion equations
to travelling front solutions, Arch. \it Rational Mech. Anal. \bf 65 \rm(1977), 335-361.


\bibitem{22} A.~Gomez and S.~Trofimchuk, Monotone traveling wavefronts of the KPP-Fisher delay equation,
\it J. Diff. Eqns, \bf 250\rm(2011), 2199-226.

\bibitem{15} K.~Hasik, J.~Kopfov\'a, P.~N\'ab$\breve{e}$lkov\'a, S.~Trofimchuk, Traveling waves in the nonlocal
KPP-Fisher
equation: different roles of the right and the left interactions, e-print arXiv:1504.06902 (2015).

\bibitem{9} Ja.~I.~Kanel, Stabilization of solutions of the Cauchy problem for equations encountered
in combustion theory, \it (Russian) Mat. Sb. \bf 59 \rm(1962), 245-288.

\bibitem{20} A.~Lorz, S.~Mirrahimi and B.~Perthame, Dirac mass dynamics in multidimensional
nonlocal parabolic equations, \it Commum Part Diff Eq, \bf 36\rm(6)(2011), 1071-1098.

\bibitem{21} A.~Lorz, T.~Lorenzi, M.~Hochberg, J.~Clairambault  and B.~Perthame, Populational adaptive evolution,
chemotherapeutic resistance and multiple anti-cancer therapies, \it ESAIM:M2AN, \bf47\rm(2)(2013), 377-399.


\bibitem{5} G.~Nadin, B.~Perthame and M.~Tang, Can a traveling wave connect two unstable states? The
case of the nonlocal Fisher equation, \it C. R. Math. Acad. Sci. Paris, \bf 349 \rm(2011), 553-557.

\bibitem{1} V.~Volpert. Elliptic partial differential equations. Volume 2. Reaction-diffusion equations.
Birkh\"{a}user, 2014.



\bibitem{2} V.~Volpert, S.~Petrovskii, Reaction-diffusion waves in biology, \it Physics of Life Reviews,
\bf 6 \rm(2009), 267-310.



\bibitem{14} A.~Volpert, V.~Volpert, V.~Volpert, Travelling Wave Solutions of Parabolic Systems,\it
Translations of Mathematical Monographs, \bf 140, \rm AMS Providence, RI, 1994.




\bibitem{17} Z.~C.~Wang, W.~T.~Li and S.~G.~Ruan, Existence and stability of travelling wave fronts in reaction
advection diffusion equations with nonlocal delay, \it J. Differential Equations, \bf 238 \rm(2007), 153-200.




\end{thebibliography}
\end{document}